\documentclass[12pt]{article}
\usepackage[psamsfonts]{eucal}
\usepackage[utf8]{inputenc}
\usepackage{amsmath,amssymb,amsthm,geometry, verbatim}
\usepackage{xcolor}

\usepackage{tikz,tikz-3dplot}
\usepackage{graphicx}
\usepackage{fullpage}
\usepackage{amsfonts}
\usepackage{microtype}
\usepackage[colorlinks=true,linkcolor=blue]{hyperref}
\hypersetup{
   citecolor=blue,
   linkcolor=blue,
}

\title{A Spectral Approach to Polytope Diameter}
\author{Hariharan Narayanan\footnote{Supported by a Swarna Jayanti fellowship}\\ TIFR Mumbai \and Rikhav Shah\\ UC Berkeley \and Nikhil Srivastava\footnote{Supported by NSF Grants CCF-1553751 and CCF-2009011.}\\ UC Berkeley}
\date{\today}

\theoremstyle{plain}
\newtheorem{theorem}{Theorem}[section]
\newtheorem{lemma}[theorem]{Lemma}
\newtheorem{definition}[theorem]{Definition}
\theoremstyle{definition}
\newtheorem{remark}[theorem]{Remark}
\newtheorem{claim}[theorem]{Claim}

\newcommand{\conv}{\textnormal{conv}}
\newcommand{\aff}{\textnormal{aff}}
\newcommand{\vol}{\textnormal{Vol}}
\newcommand{\dist}{\textnormal{dist}}

\newcommand{\spn}{\textnormal{span}}
\newcommand{\herm}[2]{\langle #1,#2\rangle}
\newcommand{\E}{\mathbb{E}}

\newcommand{\R}{\mathbb{R}}
\renewcommand{\P}{\mathbb{P}}

\newcommand{\cond}[2]{\left[ #1 \bigg| #2\right]}

\newcommand{\Z}{\mathbb{Z}}

\newcommand{\qu}{\chi_2}
\newcommand{\poly}{\mathrm{poly}}
\newcommand{\st}{:}

\newcommand{\facets}{\mathcal{F}}

\begin{document}

\maketitle

\begin{abstract}

We prove upper bounds on the graph diameters of polytopes in two settings.
The first is a worst-case bound for polytopes defined by integer constraints
in terms of the height of the integers and certain subdeterminants of the
	constraint matrix, which in some cases improves previously known results.
The second is a smoothed analysis bound: given an appropriately normalized polytope, we add small Gaussian noise to each constraint.  We consider a natural geometric measure on the vertices of the perturbed polytope (corresponding to the mean curvature measure of its polar) and show that with high probability there exists a ``giant component'' of vertices, with measure $1-o(1)$ and polynomial diameter.
Both bounds rely on spectral gaps --- of a certain Schr\"odinger operator in the first case, and a certain continuous time Markov chain in the second --- which arise from the log-concavity of the volume of a simple polytope in terms of its slack variables.
\end{abstract}

\tableofcontents
\section{Introduction}
\newcommand{\NNN}{J_{\mathsf{avg}}}
The polynomial Hirsch conjecture asks whether the diameter of an arbitrary bounded polytope $P=\{x\in\R^d:Ax\le b\}$ is at most a fixed polynomial in $m$ and $d$. This conjecture is widely open, with the best known upper bounds being $(m-d)^{\log_2 d- \log_2\log d +O(1)}$ (\cite{sukegawa2019asymptotically}, see also \cite{kalai1992quasi, todd2014improved}) and $O(m)$ for fixed $d$ (\cite{larman1970paths,barnette1974upper}); the best known lower bound  is $(1+\epsilon)m$ for some $\epsilon>0$ when $d$ is sufficiently large \cite{santos2012counterexample}.  Given this situation, there has been interest in the following potentially easier questions:
\begin{enumerate}
    \item [Q1.] Assuming $A,b$ have integer entries, bound the diameter of $P$ in terms of their size.
    \item [Q2.] Assuming $A,b$ are sampled randomly from some distribution, bound the diameter of $P$ with high probability.
\end{enumerate}
Progress on these questions (\cite{bonifas2014sub,brunsch2013finding,eisenbrand2017geometric,dadush2016shadow}, \cite{borgwardt2012simplex,spielman2004smoothed,vershynin2009beyond,dadush2020friendly}) has relied mostly on techniques from polyhedral combinatorics, integral geometry, probability, and operations research (e.g., analysis of the simplex algorithm and its cousins). 

On the other hand, the Brunn-Minkowski theory of polytopes has developed largely
separately over the past century, with several celebrated achievements including
the Alexandrov-Fenchel inequality \cite{alexandroff1937theorie} and more
generally the Hodge-Riemann relations for certain algebras associated with
simple polytopes \cite{timorin1999analogue}. One consequence of this theory is
that a certain Schr\"odinger operator (weighted adjacency matrix plus diagonal)
associated with the graph of {\em every} bounded polytope has a spectral gap
\cite{izmestiev2010colin} (see Definition \ref{defn:hessian} and Theorem
\ref{lem:one_positive_eig}). We use this fact to make progress on Q1 and Q2. In
the first setting, we show the following theorem, where $\|\cdot\|_\infty$
denotes the maximum magnitude entry of a matrix. 
\begin{theorem}\label{thm:q1}
Suppose $P=\{x\in\R^d:Ax\le b\}$ is a bounded polytope with integer coefficients
	$A\in\Z^{m\times d},b\in\Z^m$ such that every $d\times d$ minor of $[A|b]$ has
	determinant bounded by $\Delta$. Then $P$ has diameter
	\begin{equation}\label{eqn:q1bound} O(d^2\Delta\|A\|_\infty\cdot \log (m\|A\|_\infty
	\|b\|_\infty \Delta))\end{equation}

\end{theorem}
Theorem \ref{thm:q1} follows from a  more geometric result (Theorem
\ref{thm:4.4}) stated in terms of the angles between the $d-2$-faces of the polar of
$P$, which is proven in Section \ref{sec:q1}. It may be contrasted with the best previously known result of this kind due to
\cite{dadush2016shadow}, who achieved a bound of $O(d^3\Delta_{d-1}^2)$, where
$\Delta_{d-1}$ is the largest $(d-1)\times (d-1)$ subdeterminant of $A$, in
particular independent of $b$. The two bounds are incomparable in general, but
as $\Delta\le d\|b\|_\infty \Delta_{d-1}$ it is seen that \eqref{eqn:q1bound}
is nearly linear in $\Delta_{d-1}$ whereas the result of \cite{dadush2016shadow} is
quadratic, which yields an improvement for large $\Delta_{d-1}$ (compared to
$\|A\|_\infty,\|b\|_\infty,\log m)$.
However, our diameter bound is nonconstructive whereas  \cite{dadush2016shadow}
show how to efficiently find a path between any two vertices of $P$; we refer
the reader to the introduction of that paper for a more thorough discussion of
previous work in this vein (originally initiated by \cite{dyer1994random,
bonifas2014sub}). At a high level, the reason we are able to save a factor of
$d$ in comparison with previous works is that they rely on combinatorial
expansion arguments, whereas we use spectral expansion, which is amenable to a
``square root'' improvement using Chebyshev polynomials, first introduced in
\cite[Theorem 3.1]{sokal1988absence}. 

Regarding Q2, the study of diameters of random polytopes began with the influential work of Borgwardt \cite{borgwardt1977untersuchungen,borgwardt2012simplex}, who considered $A$ with i.i.d. standard Gaussian entries and $b=1$. Borgwardt showed the following ``for each'' guarantee:  for any fixed objective functions $c,c'\in\R^d$, the combinatorial distance between the vertices $x,x'$ of $P$ maximizing $\langle c,x,\rangle,\langle c',x'\rangle$ is  at most $O(d^{3/2}\sqrt{\log m})$ in expectation, provided $m\rightarrow \infty$ sufficiently rapidly. This type of result was extended to the ``smoothed unit LP'' model by Spielman and Teng in the seminal work \cite{spielman2004smoothed}; in this model one takes
\begin{equation}\label{eqn:smoothedlp} P = \{x\in\R^d: \langle x,v_j\rangle\le 1\}\end{equation}
where $v_j\sim N(a_j,\sigma^2)$ for some fixed vectors $a_1,\ldots,a_m$ normalized to have $\|a_j\|\le 1$. The original $\poly(m,d,\sigma^{-1})$ path length bound of \cite{spielman2004smoothed} was improved and simplified in \cite{deshpande2005improved,vershynin2009beyond,dadush2020friendly}; a key ingredient in each of these results was a ``shadow vertex bound'' analyzing the expected number of vertices of a two-dimensional projection of $P$. Note that all of these results provide ``for each'' guarantees: at best they bound the distance between a single pair of vertices, not between all pairs.\\

Our second contribution is to prove that for the smoothed unit LP model, {\em most} pairs of vertices in $P$ are polynomially (in $m,d,\sigma^{-1}$) close with high probability, where most is defined with respect to a certain locally defined measure on the vertices known as the {\em mean curvature measure} $\qu$ in convex geometry (see \cite{schneider2014convex,schneider1994polytopes}; we recall the definition in Section \ref{sec:smoothed}). In the language of random graph theory, this means that the graph of $P$ likely contains a ``giant component'' with respect to $\qu$ which is of small diameter.
\begin{theorem}\label{thm:q2} Assume $P$ is a random polytope sampled from the smoothed LP model. Let $\chi_2$ denote the mean curvature measure on the facets of $P^\circ$, which corresponds naturally to a measure on the set of vertices of $P$, denoted $\Omega$. 
 Then with probability at least $1-1/\poly(m)$, for every $\psi>0$ there is a subset $G:=G(\psi)\subset\Omega$ with $\qu(G)\ge (1-\psi)\qu(\Omega)$ such that the vertex diameter of $G$ is at most 
 \begin{equation}\label{eqn:q2bound} O\left(\frac{\poly(m,d)}{\sigma^4\psi}\right).\end{equation}
\end{theorem}
We prove Theorem \ref{thm:q2} in Section \ref{sec:roundedness}, where we deduce it from a more refined theorem (Theorem \ref{thm:5.15}, which includes explicit powers of $m,d$) for a certain class of well-rounded polytopes. The idea of the proof is to consider a certain continuous time Markov chain whose states are the vertices of $P$. This chain automatically has a large spectral gap by Theorem \ref{lem:one_positive_eig} and the main challenge is to bound its average transition rate. This is carried out in Sections \ref{sec:jumprate}-\ref{sec:quadrature} and involves further use of the Alexandrov-Fenchel inequalities, tools from integral geometry, Gaussian anticoncentration, and an application of the shadow vertex bound of \cite{dadush2020friendly}.

\begin{remark}
It was pointed out to us by an anonymous referee and by D. Dadush
	that there is a "folklore" result that the average distance between a
	random pair of vertices (chosen by optimizing two uniformly random
	objective functions) of $P$ as above is polynomial in $m,d,\sigma^{-1}$; this is
	seen by a Fubini type argument and the shadow vertex bounds of
	\cite{spielman2004smoothed,
	deshpande2005improved,vershynin2009beyond,dadush2020friendly}. Our
	result is incomparable since it considers a different measure on the
	vertices.
\end{remark}

\begin{remark}[Expansion of Polytopes] There has been a sustained interest in studying the expansion of graphs of combinatorial polytopes beginning with \cite{mihail1992expansion} which  conjectured that all $0/1$ polytopes have expanding graphs. The recent breakthrough \cite{anari2019log} resolved this conjecture for the special case of matroid polytopes using techniques related to high dimensional expanders and the geometry of polynomials, which may be described as capturing ``discrete log-concavity''. The present work, in contrast, uses ``continuous log-concavity'' (stemming from the Brunn-Minkowski inequality) to control the spectral gaps of certain matrices associated with the graphs of polytopes with favorable geometric properties. 

We note that the Hirsch conjecture is already known to hold for $0/1$ polytopes \cite{naddef1989hirsch}.
\end{remark}

\paragraph{Preliminaries and Notation.}
We recall some basic terminology and facts regarding polytopes; the reader may consult \cite[Chapter 4]{schneider2014convex} for a more thorough introduction. 

We denote the convex hull of a set of points by $\conv(\cdot)$ and its affine hull by $\aff(\cdot)$. 
Let $P=\{x\in\R^d:Ax\le b\}$ with $A\in\R^{m\times d}, b\in\R^m_{>0}$ be a bounded polytope containing the origin in its interior. Its polar is the polytope
$$ P^\circ = \conv\{b_j^{-1}a_j\}_{j=1}^m=:K,$$
where $a_1^T,\ldots,a_m^T$ are the rows of $A$.

A polytope in $\R^d$ is called {\em simple} if each of its vertices is contained
in exactly $d$ (codimension-$1$) facets, and {\em simplicial} if each codimension-$1$ facet contains exactly $d$ vertices. Unless otherwise noted, ``facet'' refers to a codimension-$1$ facet. The polar of a simple polytope is simplicial and vice versa.

The $1-$dimensional faces of a polytope are called {\em edges}, and are all line segments when it is bounded. The {\em vertex diameter} of a bounded polytope $P$ is the diameter of the graph of its vertices and edges. Two facets of a polytope are {\em adjacent} if their intersection is a $(d-2)$-face of the polytope. The {\em facet diameter} of a  polytope $K$ is the diameter of the graph with vertices given by its facets and edges given by the adjacency relation on facets. By duality, the vertex diameter of a simple polytope $P$ is equal to the facet diameter of $P^\circ$.

\newcommand{\hdist}{\mathrm{hdist}}
We use $\dist(\cdot,\cdot)$ to denote the Euclidean distance between two subsets of $\R^d$, and
$$\hdist(L,K):=\max\left\{\sup_{x\in L}\dist(x,K),\sup_{y\in K}\dist(y,L)\right\}$$
to denote the {\em Hausdorff distance} between two sets.

We use $V(K[j],L[d-j])$ to denote the mixed volume of $j$ copies of $K$ and $d-j$ copies of $L$ for convex bodies $K,L\subset \R^d$. The {\em Alexandrov-Fenchel} inequalities imply that these are log-concave, in the sense that for $j_1,j_2,j=\beta j_1+(1-\beta)j_2$ integers in $\{0,\ldots,d\}$ with $\beta\in [0,1]$ then
\begin{equation}\label{eqn:af} V(K[j],L[d-j])\ge V(K[j_1],L[d-j_1])^\beta \cdot V(K[j_2],L[d-j_2])^{1-\beta}.\end{equation}
We use $C$ to denote absolute constants whose value may change from line to line, unless specified otherwise.

\section{Eigenvalues of the Hessian and Spectral Gaps}
\newcommand{\p}{K}
\newcommand{\slacks}{c}
\newcommand{\A}{M}
In this section, we recall that a certain matrix associated with every bounded polytope has exactly one positive eigenvalue. 
\begin{definition}[Formal Hessian]
\label{defn:hessian}
For $K$ a bounded polytope containing the origin in its interior with $N$ facets labeled $\{1,\ldots,N\}$, let $H(K)$ denote the $N\times N$ matrix with entries
\begin{equation}\label{eqn:hdef}
(H(K))_{ij}=\begin{cases}
|F_{ij}|\csc(\theta_{ij}) & i\neq j\\
-\sum_k|F_{ik}|\cot(\theta_{ik}) & i = j
\end{cases}
\end{equation}
where $F_{ij}$ is the intersection of facets $i$ and $j$, and $\theta_{ij}$ is the angle between the vectors normal to those faces, facing away from the origin.
\end{definition}

When $K$ is simple, $H(K)$ is the Hessian of the volume of $\p(\slacks)=\{x\,|\, \A x\le\slacks\}$ with respect to the slack vector $c>0$. (see \cite[Chapter 4]{schneider2014convex}). Log-concavity of the volume implies that this Hessian has exactly one positive eigenvalue. Izmestiev \cite{izmestiev2010colin} has shown via an approximation argument that this remains true for the formal Hessian of any polytope.

\begin{theorem}[Theorem 2.4 of \cite{izmestiev2010colin}]
\label{lem:one_positive_eig}
$H(\p)$ has exactly one positive eigenvalue for any bounded polytope $\p$.
\end{theorem}
We include a self-contained proof of Theorem \ref{lem:one_positive_eig} in the Appendix for completeness\footnote{Our proof yields a slightly stronger conclusion regarding continuity of the formal Hessian than \cite{izmestiev2010colin}.}.

\newcommand{\HL}{D}
We will apply Theorem \ref{lem:one_positive_eig} to certain matrices derived from the formal Hessian and the following diagonal scaling, which plays an important role in the remainder of the paper.
\begin{definition}
Let $K,F_{ij},\theta_{ij}$ be as in Definition \ref{defn:hessian}.  Then let $D(K)$ denote the $N\times N$ positive diagonal matrix with entries
$(D(K))_{ii}=\sum_kF_{ik}\tan(\theta_{ik}/2)$.
Note that $\theta_{ik}\neq \pi$ whenever $F_{ik}=0$ since parallel facets of a convex polytope cannot intersect.
\end{definition}
\begin{lemma}[Spectral Gaps from Log-Concavity]
\label{lem:spectral_gap}
Let $K$ be a polytope and take $H:=H(K),\HL:=\HL(K)$. Let $L$ be the graph Laplacian with entries:
\begin{equation}\label{eqn:laplacian}
L_{ij}=\begin{cases}
-F_{ij}\csc(\theta_{ij}) & i\neq j\\
\sum_k F_{ik}\csc(\theta_{ij}) & i = j
\end{cases}.
\end{equation}
Then
\begin{enumerate}
\item $\HL^{-1/2}H\HL^{-1/2}$ has exactly one eigenvalue at $1$ with the rest of the eigenvalues in $(-\infty,0]$. The eigenvector corresponding to this eigenvalue is $\HL^{1/2}\mathbf1$.
\item $-\HL^{-1}L$ has exactly one eigenvalue at zero, with the rest of the eigenvalues in $(-\infty,-1]$. The left corresponding to this eigenvalue is $\HL\mathbf1$.
\end{enumerate}
\end{lemma}
\begin{proof}
Observe that $H$ is ``nearly'' a graph Laplacian in the sense that:
\begin{equation}H=-L+\HL\end{equation}
where we have used the identity
$\frac{1-\cos{\theta}}{\sin\theta}=\tan(\theta/2)$.
By Sylvester's inertia law, the signature of $H$ matches that of
\begin{equation}\label{eqn:hesslaplacian}\HL^{-1/2}H\HL^{-1/2}=-\HL^{-1/2}L\HL^{-1/2}+I,\end{equation}
which must therefore have exactly one positive eigenvalue by Theorem \ref{lem:one_positive_eig}.
However, $L\succeq 0$ and $L\mathbf1=0$, so by Sylvester's law $-\HL^{-1/2}L\HL^{-1/2}\preceq 0$ with at least one eigenvalue equal to zero. Thus, $\HL^{-1/2}H\HL^{-1/2}$ has exactly one eigenvalue equal to one, with eigenvector $\HL^{1/2}\mathbf1$ and the rest of the eigenvalues nonpositive, establishing the first claim. The second claim follows from \eqref{eqn:hesslaplacian} and the similarity of $\HL^{-1}L$ and $\HL^{-1/2}L\HL^{-1/2}$.
\end{proof}
\section{Diameter in Terms of Angles and Bit Length}\label{sec:q1}
In this section we use the spectral gap bound of Lemma \ref{lem:spectral_gap}(1)
to give a bound on the diameter of a polytope specified by integer constraints.
We begin by slightly generalizing the argument of \cite{sokal1988absence,
chung1989diameters, van1995eigenvalues}, who used Chebyshev polynomials to
control the diameter of regular (nonnegative weighted) graphs in terms of their spectra, to handle the matrix $\HL^{-1/2}H\HL^{-1/2}$ by appropriately controlling its negative entries and top eigenvector.

\begin{lemma}[Diameter in terms of Spectrum]
\label{lem:chebyshev}
Let $A$ be a weighted real symmetric adjacency matrix (possibly with self-loops
	and negative weights) for a graph $G$ on $N$ vertices.  Suppose for some
	$g>0$ there is exactly one eigenvalue of $A$ at $1+g$ with corresponding eigenvector $v$, the smallest absolute entry of which is $v_{\min}$.  Further suppose that the rest of the eigenvalues of $A$ are in the interval $[-1,1]$.  Then the diameter of $G$ is at most

	$$\frac{2\log(2N/v_{\min}^2)}{\sqrt{g}}.$$
\end{lemma}
\begin{proof}
Note that if $M\in\text{span}(I,A,\cdots,A^{k})$ then $e_i^TMe_j\neq0$ implies
	that there is a path in $G$ from $i$ to $j$ of length at most $k$.
	To this end, consider $T_k(A)$ where $T_k$ is the degree $k$ 
	Chebyshev Polynomial of the first kind.  If we find that $T_k(A)\neq0$ entry-wise, then we can conclude the diameter of $G$ is at most $k$.
Let $$A=vv^T(1+g)+\sum_{i=2}^N u_iu_i^T\lambda_i$$ be the spectral decomposition of $A$.  Let $|\cdot|$ denote the entry-wise absolute value.  Then
\begin{align*}
|T_k(A)-vv^TT_k(1+g)|
=|\sum_{i=2}^N u_iu_i^TT_k(\lambda_i)|
\le\sum_{i=2}^N|u_iu_i^T|\,|T_k(\lambda_i)|
\le\sum_{i=2}^N|u_iu_i^T|
\le N,
\end{align*}
	since $|T_k(x)|\le 1$ on $[-1,1]$.
We would therefore have $T_k(A)\neq 0$ entry-wise if $N$ is smaller then the smallest absolute entry of $vv^TT_k(1+g)$, which is lower bounded by 
	$$v_{\min}^2T_k(1+g)\ge
	\frac{v_{\min}^2}{2}(1+g+\sqrt{(1+g)^2-1})^k\ge\frac{v_{\min}^2}{2}(1+\sqrt{2g})^k.$$ 
	It suffices to pick 
	$$k\ge \frac{\log(2N/v_{\min}^2)}{\log(1+\sqrt{2g})},$$
	which is implied by taking 
	$$k= \frac{2\log(2N/v_{\min}^2)}{\sqrt{g}}$$
	and using
	$\log(1+\sqrt{2g})\ge \sqrt{2g}-g+2\sqrt{2}g^3/3\ge \sqrt{g}/2$ for $g(0,1/2]$; if
	$g\ge 1/2$ we may replace $A$ by $A/2g$ (which does not violate the hypotheses of the Lemma) and reach the desired
	conclusion.

\end{proof}

\begin{theorem}\label{thm:4.4}
Let $P=\{x\in\R^d:Ax\le b\}$ be a bounded polytope containing the origin with
	$m\ge d$, $A\in \Z^{m\times d},b\in\Z^{m}$.
	Assume all angles between pairs of adjacent facets of $P^o$ are
	contained in $[\theta_0, \pi-\theta_0]$ and  let the magnitude of the
	largest $d\times d$ subdeterminant of $[A|b]$ be $\Delta$.
Then the vertex diameter of $P$ is at most
	\[\frac{3d\log m+ O\left(d\log(d\|A\|_\infty\|b\|_\infty\Delta)\right) }{\sin(\theta_0/2)}.\]
\end{theorem}
\begin{proof}
Put $\HL:=\HL(P^o)$ and $H:=H(P^o)$.  By Lemma \ref{lem:spectral_gap}, $\HL^{-1/2}H\HL^{-1/2}$ is real symmetric with one eigenvalue at $1$ and the rest at most $0$. We can bound its smallest eigenvalue by using Lemma \ref{lem:gershgorin} and considering the similar matrix $\HL^{-1}H$.  We upper bound the absolute row sum of the $i$th row of $\HL^{-1}H$ by
\[
\sum_j|(\HL^{-1}H)_{ij}|
\le
\frac
{\sum_{i\sim j}F_{ij}\csc(\theta_{ij})}
{\sum_{i\sim k}F_{ik}\tan(\theta_{ik}/2)}
+
\frac
{\sum_{i\sim j}F_{ij}|\cot(\theta_{ij})|}
{\sum_{i\sim k}F_{ik}\tan(\theta_{ik}/2)}
\le
\frac
{2\sum_{i\sim j}F_{ij}\csc(\theta_{ij})}
{\sum_{i\sim k}F_{ik}\tan(\theta_{ik}/2)}.
\]
Taking the supremum of the above expression gives
\[
\sup_i
\frac
{2\sum_kF_{ik}\csc(\theta_{ik})}
{\sum_kF_{ik}\tan(\theta_{ik}/2)}
\le
\sup_{i\sim j}
\frac
{2\csc(\theta_{ij})}
{\tan(\theta_{ij}/2)}
=
\sup_{i\sim j}
\csc^2(\theta_{ij}/2).
\]
Therefore by Lemma \ref{lem:gershgorin}, the smallest eigenvalue of $\HL^{-1}H$, and consequently of $\HL^{-1/2}H\HL^{-1/2}$ is at least $-\csc^2(\theta_0/2)$.  Then
\[
M=\frac
{\HL^{-1/2}H\HL^{-1/2}+\csc^2(\theta_0/2)}
{\csc^2(\theta_0/2)}I
\]
has exactly one eigenvalue at $\frac
{1+\csc^2(\theta_0/2)}
{\csc^2(\theta_0/2)}=1+\csc^{-2}(\theta_0/2)$ with the rest contained in the
	interval $[0,1]$.  We can apply Lemma \ref{lem:chebyshev} with
	$g=\csc^{-2}(\theta_0/2)$ to obtain a diameter of at most:
\[
	\frac{d\log(m)+ 2\log(v_{\min}^{-1})}{\sin(\theta_0/2)}.
\]
The eigenvector $v$ corresponding to eigenvalue $1+\csc^2(\theta_0/2)$ is simply $\mathbf{1}^T\HL^{1/2}$ normalized, so
\[
v_{\min}
=\frac{\min_i(\mathbf1^T\sqrt{\HL})_i}{||\mathbf1^T\sqrt{\HL}||_{2}}
\ge\frac1{\sqrt{N}}\sqrt{\frac{\min_i\HL_{ii}}{\max_i\HL_{ii}}}
\ge\frac1{\sqrt N}\frac
{\min_i\sum_kF_{ik}\tan(\theta_{ik}/2)}
{\max_i\sum_kF_{ik}\tan(\theta_{ik}/2)}
	\ge\frac{\sin^2(\theta_0)}{4N^{3/2}}\frac{\min_{i,j}F_{ij}}{\max_{i,j}F_{ij}},
\]
	where we used that $\theta_{ik}\in [\theta_0,\pi-\theta_0]$ implies $\sin(\theta_0/2)\le \tan(\theta_{ik}/2)\le
	\frac{1}{\sin(\theta_0/2)}$ and $\sin(x/2)\ge \sin(x)/2$.
	Finally use $N\le \binom{m}{d}\le  m^d/2$  as well as the estimates from
	Lemma \ref{lem:ratio_of_volumes} to see that
	\begin{align*}
		\log (v_{\min}^{-1}) &\le \log(4N^{3/2}) + 2\log(2d\|A\|_\infty
	\Delta)+d\log(\sqrt{d}\|A\|_\infty)+d\log d+2d\log(\|b\|_\infty)\\
		&\le 2d\log(m)+ O(d\log(d\|A\|_\infty\|b\|_\infty) +
		\log(\Delta)),
	\end{align*}
yielding the advertised bound.
\end{proof}

\begin{lemma}[Gershgorin's circle theorem]
\label{lem:gershgorin}
The smallest (real) eigenvalue of $M$ is at least $-\sup_i\sum_j|M_{ij}|$.
\end{lemma} 

\begin{lemma}[Worst Case Volumes and Angles]
\label{lem:ratio_of_volumes}
Let $P^o=\conv(a_1/b_1,\cdots,a_m/b_m)$ be a polytope where each
	$a_i/b_i\in\R^d$ is a vertex and $a_i\in\Z^d,b_i\in\Z$.
Then:
\begin{enumerate}
    \item 
	    The smallest co-dimension $2$ face of $P^o$ has volume at least
		$1/(d!\|b\|_\infty^{2d})$, and the largest co-dimension $2$ face
		has volume at most $(\sqrt{d}\|A\|_\infty)^d$. 
	\item If the largest $d\times d$ minor of $[A|b]$ is bounded in
		magnitude by $\Delta$, then the angle between any two
		adjacent facets of $P^o$ satisfies
		$\csc(\theta)\le 2d\Delta\|A\|_\infty.$
\end{enumerate}
\end{lemma}
\begin{proof}
Every co-dimension $2$ face can be written as the convex hull of some subset of size at least $d-1$ of the vertices $a_1/b_1,\cdots,a_m/b_m$.  Without loss of generality, say that $F=\conv(a_1/b_1,\cdots,a_{d-1}/b_{d-1})$ is the smallest co-dimension $2$ face.  Then its volume is:
\[\vol(\conv(a_1/b_1,\cdots,a_{d-1}/b_{d-1}))
=
\frac1{(d-2)!}\sqrt{|\det (M^TM)|}\ge \frac{1}{d!(b_1\ldots
	b_{d-2})b_{d-1}^{d-2}}\ge \frac{1}{d!\|b\|_\infty^{2d}},
\]
where $M$ is the $d\times (d-2)$ matrix whose $i$th column is $a_i/b_i-a_{d-1}/b_{d-1}$, and we have used that the determinant of a nonsingular integer matrix is at least one. 
On the other hand, $P^o$ is contained inside the $\ell_{2}$ ball of radius
	$d^{1/2}\|A\|_\infty$, and so each co-dimension $2$ face of $P^o$ is
	contained in a cross section of that ball, and consequently has volume
	at most $(\sqrt{d}\|A\|_\infty)^{d}$, establishing (1).

Regarding the angles,  consider without loss of generality two adjacent facets
	of $P^o$, with vertices numbered so that 
	$F=\conv(a_1/b_1,\ldots,a_d/b_d)$ and 
	$F'=\conv(a_2/b_2,\ldots,a_d/b_d,a_j/b_j)$, $j>d$, and $|b_j|\le |b_1|$.
 Observe that the angle $\theta$ between the normals to these adjacent facets satisfies:
$$ \csc(\theta) = \frac{\dist(a_j/b_j,F)}{\dist(a_j/b_j,\aff(F))}.$$
The numerator is at most the distance between $a_j/b_j$ and any vertex of $F$, which is at most 
	$$\|a_1/b_1-a_j/b_j\|_2\le
	\sqrt{d}\frac{\|a_1\|_\infty+\|a_j\|_\infty}{|b_j|}\le
	\frac{2\sqrt{d}\|A\|_\infty}{|b_j|},$$
	by our choice of $j$.
The denominator is given by 
\begin{align*}
	\dist(a_j/b_j,\aff(F)) &\ge \dist((a_j/b_j,1)^T
	,\mathrm{span}(\hat{a_1},\ldots\hat{a_d}))\quad\textrm{where
	$\hat{a_i}:=(a_i, b_i)^T\in\R^{d+1}$}\\
	&=\frac{1}{|b_j|\|e_j^TM^{-1}\|}
\end{align*}
	where $M$ is the $(d+1)\times (d+1)$ matrix with columns
	$(\hat{a_1},\ldots,\hat{a_d},\hat{a_j})$, which must be invertible since $\conv(a_1/b_1,\ldots,a_d/b_d,a_j/b_j)$ is a full dimensional simplex. By the adjugate formula, the entries of $M^{-1}$ are of magnitude at most $\Delta$, so we have
$$\dist(a_j,\aff(F))\ge \frac{1}{\sqrt{d}|b_j|\Delta}.$$
Combining these bounds and cancelling the $|b_j|$ yields 
	$$\csc(\theta)\le 2d\|A\|_\infty \Delta,$$ establishing (3).

\end{proof}

Finally, we can prove the bound advertised in the introduction.
\begin{proof}[Proof of Theorem \ref{thm:q1}]
	Applying Theorem \ref{thm:4.4}, Lemma \ref{lem:ratio_of_volumes}(2), and
	$\sin(x/2)\ge \sin(x)/2$, we find that the diameter of $P$ is at most 
	$$ O(d\log (m\|A\|_\infty \|b\|_\infty \Delta)\times
	d\|A\|_\infty\Delta,$$
	as advertised.
\end{proof}
\section{Smoothed Analysis}\label{sec:smoothed}
\newcommand{\all}{\binom{[m]}{d}}
\newcommand{\B}{\mathcal{B}}
\newcommand{\Ce}{\mathcal{C}}
\renewcommand{\L}{\mathcal{L}}
\newcommand{\sa}[1]{\mathrm{Vol}_{d-1}(\partial #1)}
\newcommand{\ghat}{\hat{\mathbf{g}}}
\newcommand{\Se}{S_\epsilon}
\newcommand{\one}{\mathbf{1}}

In this section we consider the  ``smoothed unit LP'' model defined in \eqref{eqn:smoothedlp}.
Suppose $P_0$ is a fixed polytope specified as $$P_0=\{x\in\R^d: \langle a_j,x\rangle \le 1, j\in [m]\},$$
for some vectors $\|a_j\|\le 1$, 
and consider the random polytope 
$$P = \{x\in\R^d: \langle v_j,x\rangle \le 1\},$$
where $v_j=a_j+g_j$ for $g_j\sim N(0,\sigma^2 I_d)$ i.i.d spherical Gaussians.  Denote the polars of $P_0$ and $P$ by
$$ K_0:=P_0^\circ = \conv(a_1,\ldots,a_m)\subset B_2^d,$$
$$K:=P^\circ = \conv(v_1,\ldots,v_m).$$

Note that $K$ is simplicial with probability one, so each of its $k$-dimensional faces has exactly $k+1$ vertices. We will use the notation
 $F_S :=\conv\{v_j:j\in S\}$
to denote faces of $K$ and
$\facets_k(K):=\left\{S\in \binom{[m]}{k+1}: \textrm{$F_S$ is a $k$-dimensional face of $K$}\right\}$ to denote the set of all faces of $K$.
The $k$-dimensional volume of a face $F_S, S\in\facets_k(K)$ will be denoted by $|F_S|$ or $\vol_k(F_S)$.
We will often abbreviate $F_{S\cap T}$ as $F_{ST}$ for adjacent $S,T$.
 For two $S,T\in \all$, let $\theta_{ST}\in (0,\pi)$ denote the angle between the unit normals $u_S, u_T$ to $F_S, F_T$, respectively; note that almost surely $\theta_{ST}\neq 0,\pi$ for every $S,T\in\all$. 

We will pay special attention to the set of $(d-1)$-faces of $K$, which we denote as
$$\Omega:=\facets_{d-1}(K)\subset \all.$$
Define the measures $\qu, \pi, \delta:\Omega\rightarrow\R_{\ge 0}$ by 
\begin{equation}\label{eqn:meancurvature} 
\qu(S):= \sum_{T\in\Omega} |F_{ST}|\theta_{ST},\end{equation}
\begin{equation}\label{eqn:pi} \pi(S):=\sum_{T\in\Omega} |F_{ST}|\tan(\theta_{ST}/2),\end{equation}
\begin{equation}\label{eqn:delta} \delta(S):= \sum_{T\in\Omega} |F_{ST}|\csc(\theta_{ST}).\end{equation}

It will be convenient to make two further {\bf technical assumptions} on $K_0$ and $\sigma$ for the proofs of our results; in Section \ref{sec:roundedness} we will show that any instance of the smoothed unit LP model may be reduced to one satisfying both assumptions with parameter
\begin{equation}\label{eqn:rvalue}
r = \Omega(\sigma m^3),
\end{equation}
incurring only a $\poly(m)$ loss in the diameter.
Let $K_0^{(j)}=\conv(a_i:i\neq j)$ be the polytope obtained from $K_0$ by deleting vertex $a_j$.
\newcommand{\aS}{\bf{(S)}}
\newcommand{\aR}{\bf{(R)}}
\begin{itemize}
    \item[\aR] Roundedness of Subpolytopes: There is an $r\in (0,1)$ such that for every $j\le m$:
    $$ rB_2^d\subset K_0^{(j)}.$$
    \item[\aS] Smallness of $\sigma$:
    \begin{equation}\label{eqn:alphadef} \alpha:= 6\sigma\sqrt{d\log m} < r/d^2.\end{equation}
\end{itemize}

The main result of this section is the following ``almost-diameter'' bound with respect to the measure $\pi$.
\begin{theorem}\label{thm:5.15} Assume $\aS, \aR$. Then with probability at least $1-1/\log(m)$, for every $\phi>0$ there is a subset $G:=G(\phi)\subset\Omega$ with $\pi(G)\ge (1-\phi)\pi(\Omega)$ such that the facet diameter of $G$ is at most $ \tilde{O}(m^3d^8/\sigma^2r\phi).$
\end{theorem}
\begin{remark}\label{rem:polyprob} The probability in Theorem \ref{thm:5.15} may be upgraded to $1-m^{-c}$ for any $c$ at the cost of an additional $m^c$ factor in the diameter bound, by applying Markov's inequality in the proof of Lemma \ref{lem:polyjump} with a different threshold.
\end{remark}

\newcommand{\pibar}{\overline{\pi}}
The proof of Theorem \ref{thm:5.15} relies on three properties of the (random) continuous time Markov chain  with state space $\Omega$ and infinitesimal generator\footnote{The reader may consult e.g. \cite[Chapter 6]{grimmett2020probability} for an introduction to continuous time Markov processes.}
\begin{equation}\label{eqn:semigroup}Q:=-\HL^{-1}L,\end{equation}
where $L$ is as in \eqref{eqn:laplacian}. The corresponding Markov semigroup
\[P(t):=\exp(-t\HL^{-1}L),\quad t\ge 0,\]
has stationary distribution proportional to $\HL\mathbf{1}=\pi(\cdot)$ by Lemma \ref{lem:spectral_gap}(2); call the normalized stationary distribution $\pibar(\cdot):=\pi(\cdot)/\pi(\Omega)$

The first  property is that the stationary distribution $\pibar$ is (in a quite mild sense) non-degenerate, with high probability. Apart from being essential in our proofs, this relates the measure $\pi$ to well-studied measures in convex geometry such as the surface measure and mean curvature measure $\qu(\cdot)$, clarifying the meaning of Theorem \ref{thm:5.15}. The proof of Lemma \ref{lem:degeneracy} appears in Section \ref{sec:nondegenerate}.
\begin{lemma}[Non-degeneracy of $\pibar$]\label{lem:degeneracy} Assume $\aS,\aR$.
With probability at least $1-1/m^2$:
\begin{enumerate}
    \item $\min_{S\in\Omega} \pibar(S)\ge \pibar_{min}:= C\frac{m^{-2d^2}r^2}{d^3}.$
    \item  $c\sa{K}\le \pi(\Omega)\le O(d^3r^{-2})\sa{K}.$
    \item For every $S\in\Omega$,
    $\qu(S)/2\le \pi(S)\le O(r^{-1})\qu(S).$
\end{enumerate}
\end{lemma}

The second property is that $Q$ (almost surely) has a spectral gap of at least one, by Lemma \ref{lem:spectral_gap}(2). This implies that the chain \eqref{eqn:semigroup} mixes rapidly to $\pibar$ (in the sense of continuous time) from any well-behaved starting distribution. 
In particular let us say that a probability measure $p$ on $\Omega$ is  an {\em $M$-warm start} if
\[\sup_{S\in\Omega}\frac{p(S)}{\pibar(S)}\le M.\]
Let $\ell_2(\pibar)$ denote the inner product space on defined on $\R^\Omega$, where the inner product is given by $\langle f, g\rangle_{\ell_2(\pibar)} := \sum_{S\in\Omega} \pibar(S) f(S) g(S)$, and let $\ell_1(\pibar)$ be the corresponding $\ell_1$ space.  Let $\Pi$ be the $\Omega\times \Omega$ diagonal matrix whose $S^{th}$ diagonal entry is $\pibar(S)$.
We define the density of $p$ with respect to $\pibar$ to be the the vector with entries  $\frac{p(S)}{\pibar(S)}$. We omit the proof of the following standard fact.
\begin{lemma}[Warm Start Mixing]  \label{lem:mixing}
If $p$ is $M-$warm, then for $\tau>0$,  $t=\Omega(\log(M/\tau))$ time, one has
\[||\pibar - pP(t)||_{TV}\le \tau.\]
\end{lemma}

The third and final property is a bound on the rate at which the continuous chain makes discrete transitions between states.
Let $\NNN$ denote the average number of state transitions made by the continuous time chain in unit time, from stationarity, and note that$$\NNN = \frac{\sum_{S\in\Omega} \pi(S) |Q(S,S)|}{\sum_{S\in\Omega}\pi(S)}= \frac{\sum_{S\in\Omega} \delta(S)}{\sum_{S\in\Omega}\pi(S)}$$
as the diagonal entries of the generator $Q$ are equal to $-\delta(S)/\pi(S)$. The most technical part of the proof is the following probabilistic bound.
\newcommand{\jumpbound}{\tilde{O} (m^3d^6/\sigma^2r)}
\begin{lemma}[Polynomial Jump Rate]\label{lem:polyjump} Assume $\aS,\aR$. With probability at least $1-1/\log(m)$, the continuous time Markov chain defined by \eqref{eqn:semigroup} satisfies:
$$\NNN\le \jumpbound.$$
\end{lemma}
The proof of this lemma involves showing that the facets of $K$ are well-shaped and have non-degenerate angles between them in a certain average sense, and is carried out in Sections \ref{sec:jumprate}, \ref{sec:perimeter}, and \ref{sec:quadrature}.

Combining these ingredients, we can prove Theorem \ref{thm:5.15}

\newcommand{\length}{\mathrm{length}}
\newcommand{\target}{\mathrm{target}}
\begin{proof}[Proof of Theorem \ref{thm:5.15}]
Let $T$ be a fixed positive time to be chosen later.
Consider the continuous time chain \eqref{eqn:semigroup}, and for $F\in\Omega$ let the random variable $J_F^T$ denote the number of transitions in $[0,T]$ when the chain is started at $F$. With probability $1-1/m^2$ we have
$$ \sum_{F\in\Omega} \pibar(F)\E J_F^T = T\NNN \le \jumpbound\cdot T$$
by Lemma \ref{lem:polyjump} so there is a facet $F_0\in\Omega$ satisfying 
\begin{equation}\label{eqn:fgood} \E J_{F_0}^T \le \jumpbound\cdot T.\end{equation}
By Lemma \ref{lem:degeneracy}(1), the distribution $\delta_{F_0}$ concentrated on $F_0$ is $\pibar_{min}^{-1}-$warm with probability $1-1/m^2$. Invoking Lemma \ref{lem:mixing} with starting distribution $\delta_{F_0}$ and parameters 
$$T=O(\log(1/\pibar_{min}))=\tilde{O}(d^2\log(1/r)),\quad M=\pibar_{min}^{-1},\quad \tau=\pibar_{min}/2$$ we have
$$ \|\pibar - \delta_{F_0}P(T)\|_{TV}\le \pibar_{min}/2.$$
Combining this with \eqref{eqn:fgood}, we obtain a distribution on discrete paths $\gamma$ in $\Omega$ (with respect to the adjacency relation $\sim$) such that each path has source $F_0$, 
$$\E \length(\gamma)\le \jumpbound\cdot T,$$
and the distribution of $\target(\gamma)$ is within total variation distance $\pibar_{min}/2$ of $\pibar$. 
Letting
$$ G = \{\target(\gamma):\length(\gamma)\le 2\E\length(\gamma)/\phi\}$$
we immediately have that the diameter of $G$ is at most 
$$ \jumpbound\cdot 2T/\phi = \tilde{O}(m^3d^8/\sigma^2r\phi)$$
and by Markov's inequality $\pibar(G)\ge 1-\phi$, as desired.

\end{proof}

Before proceeding with the proofs of Lemmas \ref{lem:degeneracy} and \ref{lem:polyjump}, we collect the probabilistic notation used throughout the sequel. We will often truncate on the following two high probability events.
Fix 
\begin{equation}\label{eqn:epsdef}\epsilon:=m^{-5d}\end{equation} and define:
$$\B:= \left\{\min_{S\in\all, j\in [m]\setminus S} \dist(v_j, \aff(F_S))\ge\epsilon\right\},$$
$$\Ce:=\left\{\max_{j\in [m]}\|g_j\|\le \alpha\right\}.$$
Note that whenever $\sigma>m^{-d}$ (which we may assume without loss of generality, as otherwise the diameter is trivially at most $1/\sigma$):
\begin{equation}\label{eqn:btrunc}\P[\B]\ge 1-O(m^{-4d}/\sigma)\ge 1-1/m^3,\end{equation} since the density of the component of $v_j$ orthogonal to $\aff(F_S)$ is bounded by $1/\sigma$ and there are at most $m^d$ facets. We also have
\begin{equation}\label{eqn:ctrunc}\P[\Ce]\ge 1-1/m^3,\end{equation}
by standard Gaussian concentration and a union bound. 

We will repeatedly use that on $\Ce$, we have the Hausdorff distance bounds
\begin{equation}\label{eqn:hdist} \hdist(K,K_0)\le\alpha,\qquad \hdist(K^{(j)},K_0^{(j)})\le \alpha\quad\forall j\le m,\end{equation}
for $\alpha$ as in \eqref{eqn:alphadef}, since if $x=\sum_{j\le m}c_j (a_j+g_j)\in K$ for some convex coefficients $c_j$ then $x_0=\sum_{j\le m} c_j a_j\in K_0$ and $\|x-x_0\|\le \alpha$.

For an index $j\in [m]$ let $\ghat_j:=(g_1,\ldots,g_{j-1},g_{j+1},\ldots g_m)$ and let $K^{(j)}=\conv(v_i:i\neq j)$. Note that $K^{(j)}$ is a deterministic function of $\ghat_j$.
Define the indicator random variables
$$ K_S := \{ F_S\in \facets_{d-1}(K)\},\qquad K_S^{(j)} := \{ F_S\in \facets_{d-1}(K^{(j)})\}$$
for subsets $S\in\all$. 
 It will be convenient to fix in advance a total order $<$ on $\all$.
 
 We will occasionally refer to
$$\sum_{F\in\facets_{k}(K)} \vol_{k}(F)$$
as the {\em $k-$perimeter} of $K$.
\subsection{Nondegeneracy of $\pibar$}\label{sec:nondegenerate}
We will repeatedly use the following fact relating Hausdorff distance and containment of convex bodies.
\begin{lemma}[Containment of Small Perturbations]\label{lem:contain} If  $\hdist(K,K_0)\le\alpha$ for any two convex bodies and $rB_2^d\subset K_0$, then
$$(1+2\alpha /r)^{-1}K_0^{}\subset K^{} \subset (1+\alpha/r)K_0.$$
\end{lemma}
\begin{proof} 
The second containment is immediate from
$$K\subset K_0+\alpha B_2^d\subset K_0+(\alpha/r) K_0.$$

The condition $\hdist(K,K_0)\le\alpha$ also implies $K_0\subset K+\alpha B_2^d$. To turn this into a multiplicative containment, we claim that $(r/2)B_2^d\subset K$. If not, there is a point $z\in \partial (r/2)B_2^d\setminus K$. Choose a halfspace $H$ supported at $z$ containing $K$. Let $y$ be a point in $\partial(rB_2^d)$ at distance at least $r/2$ from $H$ and note that $y\in K_0$. But now $\dist(y,K)\ge \dist(y,H)\ge r/2>\alpha$, violating that $K_0\subset K+\alpha B_2^d$. Thus, we conclude that $K_0\subset (1+2\alpha/r)K$, establishing the first containment. 
\end{proof}
\begin{proof}[Proof of Lemma \ref{lem:degeneracy}] Condition on $\Ce$.
By $\aS, \aR$, \eqref{eqn:hdist}, and Lemma \ref{lem:contain}, we have
\begin{equation}\label{eqn:krounded}K\supset (1+2/d^2)^{-1}K_0\supset (r/2)B_2^d,\end{equation}
and also $K\subset (1+\alpha)B_2^d$. Consequently, the angle between any two adjacent facets $F_S,F_T$ of $K$ must satisfy
$$ |\theta_{ST}-\pi|=\Omega(1/r),$$
which implies
\begin{equation}\label{eqn:tanbound} \theta_{ST}/2\le \tan (\theta_{ST}/2) \le O(1/r)\theta_{ST}\end{equation}
for all $\theta_{ST}$. Thus, for each facet $S\in\Omega$:
\begin{equation}\label{eqn:qubound} \qu(S)/2\le \pi(S)\le O(r^{-1})\qu(S),\end{equation}
establishing Lemma \ref{lem:degeneracy}(3).

Equation \eqref{eqn:krounded} further implies:
$$\frac{|\partial K|}{|K|}\leq \frac{2d}{r}.$$ 
By e.g. \cite[Section 4.2]{schneider2014convex}, we have the quermassintegral formulas:
\begin{equation}\label{eqn:quer1} d\cdot V(K[d-1],B_2^d[1]) = \sum_{S\in\Omega}|F_S|=|\partial K|,\end{equation}
\begin{equation}\label{eqn:quer2}\binom{d}{2}V(K[d-2],B_2^d[2]) = \sum_{S<T\in\Omega} |F_{ST}|\theta_{ST}.\end{equation}
By the Alexandrov-Fenchel inequality with $\beta=1/2$: \begin{eqnarray*} \qu(\Omega)={d \choose 2}V(K[d-2], B_2^d[2]) &\leq & {d \choose 2}  \frac{V(K[d-1],B_2^d[1])^2}{V(K[d])} \\&\leq& {d \choose 2} \frac{  |\partial K|^2}{|K|}\\ & \leq & O(d^3)\frac{|\partial K|}{r}\\
& \leq & O(d^3/r)|\partial K|.\end{eqnarray*}

By Alexandrov-Fenchel with $\beta=1/(d-1)$, we also have
\begin{eqnarray*}
\qu(\Omega)={d \choose 2} V(K[d-2], B_2^d[2]) & \geq &  
{d\choose 2} V(K[d-1],B_2^d[1])^{\frac{d-2}{d-1}} V(B_2^d[d])^{\frac{1}{d-1}} \\
&=& {d\choose 2} \left(d^{-1} |\partial K|\right)^\frac{d-2}{d-1} |B_2^d|^\frac{1}{d-1}\\
& \ge & {d\choose 2}\sqrt{\frac{2\pi e}{d^3}} |\partial K|^\frac{d-2}{d-1}\\
& \ge & C|\partial K|.\end{eqnarray*}
The last step follows from the fact that $K\subset (1+\alpha)B_2^d$ so $|\partial K|^{\frac{1}{d-1}}=O(1)$.
Combining these inequalities with \eqref{eqn:tanbound},\eqref{eqn:qubound}, we conclude that:
\begin{equation}\label{eqn:pisum} C|\partial K|\le \pi(\Omega)\le O(d^3r^{-2}) |\partial K|,\end{equation}
establishing Lemma \ref{lem:degeneracy}(2).

The event $\B$ implies that for every $S\in\Omega$:
\begin{align*}
    \pi(S) &= \sum_{T\sim S} |F_{ST}|\tan(\theta_{ST}/2)\\
    &\ge C\epsilon \sum_{T\sim S} |F_{ST}|\quad\textrm{since $\tan(\theta_{ST}/2)\ge C\epsilon$}\\
    &\ge C\epsilon (d-1) |F_S|^\frac{d-2}{d-1}|B_2^d|^\frac{1}{d-1}\quad\textrm{by the isoperimetric inequality}\\
    &\ge C\epsilon d |F_S|,
\end{align*} 
where in the last step we used $|F_S|=O(1)$.
Conditional on $\B$ Lemma \ref{lem:inradius} implies that $|F_S|\ge \frac{\epsilon^{d-1}}{d}$ for every $S\in\Omega$, so we conclude that
$$ \pi(S)\ge C\epsilon^d\quad\forall S\in\Omega,$$
and consequently by \eqref{eqn:pisum}
$$ \pibar(S)\ge \frac{C\epsilon^d r^2}{d^3},$$
yielding Lemma \ref{lem:degeneracy}(1), as desired.

\end{proof}
The proof of the following easy Lemma is omitted in this conference version.
\begin{lemma}[Inradius of a Simplex]\label{lem:inradius}
If $L=\conv(v_1,\ldots,v_{t+1})$ is a $t$-dimensional simplex such that each vertex of $L$ is at distance $s$ from the affine span of the remaining vertices, then $L$ contains a ball of radius $s/({t+1})$.\end{lemma}

\subsection{Average Jump Rate Bound}\label{sec:jumprate}
\newcommand{\denomconst}{O(1)}
\newcommand{\numconst}{O(d\log m/\sigma)}
\newcommand{\perimeterconst}{O(m^2d^{9/2}\log^{5/2}(m)/\sigma r)}

In this section we establish the following Lemma, which immediately implies Lemma \ref{lem:polyjump} by $\P[\B\Ce]\ge 1-2/m^3$ and Markov's inequality applied to the expectation below (absorbing the $\log(m)$ factor into the $\tilde{O}$).
\begin{lemma}[Main Estimate]\label{lem:main} Assume $\aS,\aR$ in the above setting. Then 
\begin{equation}\label{eqn:jrate} \E \frac{\sum_{S\in\Omega} \delta(S)}{\sum_{S\in\Omega}\pi(S)}\cdot \B\Ce \le \jumpbound.\end{equation}
\end{lemma}
\begin{proof}
\begin{align*}
    \E \frac{\sum_{S\in\Omega} \delta(S)}{\sum_{S\in\Omega}\pi(S)}\cdot \B\Ce &\le
    \denomconst\cdot \E \frac{\sum_{S\in\Omega} \delta(S)}{\vol_{d-1}(\partial K_0)}\cdot \B\Ce\quad\textrm{by \eqref{eqn:pisum} and \eqref{eqn:krounded}}\\
    &=\frac{\denomconst}{\vol_{d-1}(\partial K_0)}\cdot \E\left[ \Ce\cdot\sum_{S<T\in\all} \B |F_{ST}|\csc\theta_{ST}K_SK_TK_{ST}\right]\\
    &\le \frac{\denomconst\cdot\numconst}{\vol_{d-1}(\partial K_0)}\E\left[ \Ce\cdot\sum_{S<T\in\all} |F_{ST}|K^{(S\setminus T)}_{ST}\right]\quad\textrm{by Lemma \ref{lem:angles}}\\
	&\le \frac{\denomconst\cdot\numconst}{\vol_{d-1}(\partial K_0)}\E\left[ \Ce\cdot\sum_{j \le m}\sum_{S<T\in\all} |F_{ST}|K^{(j)}_{ST}\right]\\
	   &\le \frac{\denomconst\cdot\numconst}{\vol_{d-1}(\partial K_0)}\cdot\perimeterconst\cdot\sum_{j\le m}\sa{K_0^{(j)}} \quad\textrm{by Lemma \ref{lem:perimeter}}\\
	   &\le  \frac{\denomconst\cdot\numconst}{\vol_{d-1}(\partial K_0)}\cdot\perimeterconst\cdot m\sa{K_0} \quad\textrm{since $K_0^{(j)}\subset K_0$}\\
	   &\le m\cdot \denomconst\cdot\numconst\cdot\perimeterconst, 
\end{align*}
implying the desired conclusion.
\end{proof}

\begin{figure}
\centering
\includegraphics{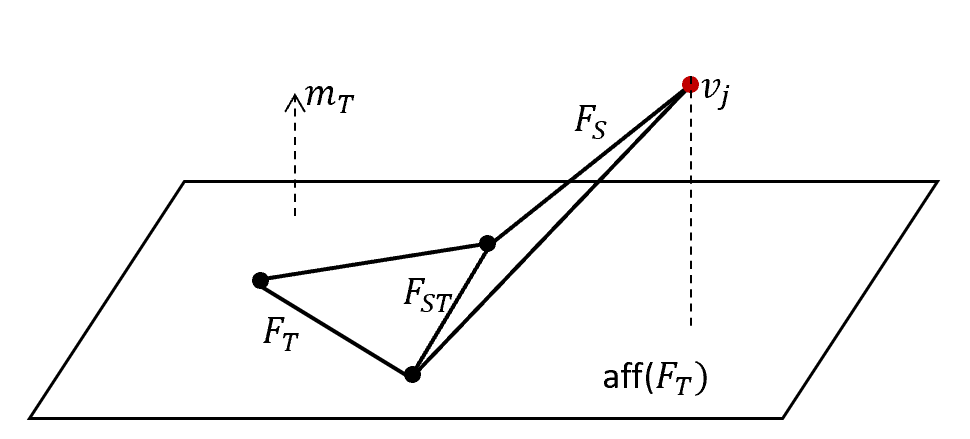}
  \caption{Proof of Lemma \ref{lem:angles}}
  \label{fig:boat1}
\end{figure}

\begin{lemma}[Angles Large On Average] \label{lem:angles}For every $S,T\in\all$ with $S\setminus T=\{j\}$:
\begin{equation}\label{eqn:oneangle} \E\cond{\B|F_{ST}| \csc\theta_{ST} K_{ST}}{\ghat_j,\Ce}\le \numconst\cdot \E \cond{|F_{ST}| K_{ST}^{(j)}}{\ghat_j,\Ce}.\end{equation}
\end{lemma}
\begin{proof}
By trigonometry,
$$ \csc\theta_{ST}=\frac{\dist(v_j,\aff(F_{ST}))}{\dist(v_j,\aff(F_T))}\le \frac{3}{\dist(v_j,\aff(F_T))},$$
conditional on $\Ce$, since $K$ has diameter at most $2+2\alpha\le3$.
The distance in the denominator can be rewritten as
\[\dist(v_j,\aff(F_T))=\dist(g_j+a_j,\aff(F_T))=\dist(g_j,\aff(F_T)-a_j)=|h_j-x_T|\]
where $h_j=\herm{g_j}{m_T}$ and $x_T=\dist(0,\aff(F_T)-a_j)\le 4$ for $m_T$ the unit normal to $\aff(F_T)$.
Moreover,
$$ |F_{ST}|K_{ST}\le |F_{ST}|K^{(j)}_{ST}$$
with probability one conditional on $\ghat_j$ since $S\cap T\in\facets_{d-2}(K)$ implies $S\cap T\in\facets_{d-2}(K^{(j)})$ as $j\notin S\cap T$. Combining these facts, the left hand side of \eqref{eqn:oneangle} is at most 
\begin{align*}
     \E \cond{ \B |F_{ST}| K^{(j)}_{ST} \frac{3}{\dist(v_j,\aff(F_T))}}{\ghat_j,\Ce} 
     &= 3|F_{ST}| K^{(j)}_{ST} \E \cond{\B \frac{1}{|h_j-x_T|}}{\ghat_j,\Ce},
\end{align*}
Notice that $h_j$ has density on $\R$ bounded by
$$t\mapsto \frac{1}{\sqrt{2\pi}\sigma}e^{-t^2/2\sigma^2}\frac1{\P[\|g_j\|\le \alpha]}\le \frac{1}{\sigma},$$
and $\epsilon\le|h_j-x_T|\le|h_j|+x_T\le 4+\alpha<5$ conditioned on $\mathcal{B},\Ce$, so the last conditional expectation is at most
$$3\int_{\epsilon}^5 \frac1{\sigma t}dt = 2(\log(1/\epsilon)+\log 5)/\sigma \le O(d\log m/\sigma),$$
completing the proof.
\end{proof}
The most technical part of the proof is the following $(d-2)$-perimeter estimate, whose proof is deferred to Section \ref{sec:perimeter}. The conceptual meaning of this estimate is that on average, the $(d-2)$-dimensional surface area of a random facet of $K^{(j)}$ is well-bounded by its $(d-1)$-dimensional volume.
\begin{lemma}[Codimension $2$ Perimeter versus Surface Area]\label{lem:perimeter} Assume $\aR,\aS$. For every $j\in [m]$:
	$$\E \Ce \sum_{S<T\in\all} |F_{ST}| K_{ST}^{(j)} \le \perimeterconst\cdot \sa{K_0^{(j)}}.$$
\end{lemma}

\subsection{Proof of Lemma \ref{lem:perimeter}}\label{sec:perimeter}
\newcommand{\Vol}{\mathrm{Vol}}

The key step in the proof is to show that for any well-rounded polytope $L_0$,
there a distribution on two-dimensional planes $W$ such that the
$(d-2)$-perimeter of every nearby polytope $L$ is accurately reflected in the
average number of vertices of $W\cap L$. Since this number of vertices is small
in expectation by \cite{dadush2020smoothed}, we can then conclude that the codimension $2$ perimeter is small.

{\em In this section and the next only},  the variable $\epsilon$ will refer to a quantity tending to zero (as opposed to the definition \eqref{eqn:epsdef}).

\newcommand{\quadbound}{\Omega\left(\frac{r_1}{d^{3/2}r_2\eta}\right)}
\begin{lemma}[Quadrature by Planes]\label{lem:quadrature} Let $r_1 B_2^d\subset L_0\subset r_2 B_2^d$, where $1 \in [r_1, r_2]$. There there is a probability distribution  on two dimensional planes $W$ in $\R^d$ such that for sufficiently small $\epsilon>0$ the following holds uniformly over every polytope $L$ with at most $m^d$ facets satisfying
\begin{equation}\label{eqn:lclose} \hdist(L,L_0)\le \eta<\frac{r_1}{2d}:\end{equation} every $(d-2)$-dimensional disk $\Se$ of radius $\epsilon$ contained in the interior of a $(d-2)$-dimensional face of $L$ satisfies
	$$\sa{L_0}\cdot \P [W\cap \Se\neq\emptyset] \ge \quadbound\cdot\vol_{d-2}(\Se).$$
Moreover, for every $(d-2)$-dimensional affine subspace $H\subset \R^d$, $\P[W\cap H > 1] = 0$.
\end{lemma}
The proof of Lemma \ref{lem:quadrature} is deferred to Section \ref{sec:quadrature}.  


We rely on the following result of Dadush and Huiberts \cite[Theorem 1.13]{dadush2020smoothed} (they prove something a little stronger, but we use a simplified bound).
\begin{theorem}[Shadow Vertex Bound]\label{thm:shadow} Suppose $W$ is a fixed two dimensional plane and $Q=\conv\{v_1,\ldots,v_m\}$ where $v_i\sim N(a_i,\sigma^2 I)$ with $\|a_i\|\le 1$. Then
\begin{equation}\label{eqn:dadush} \E [|\facets_0(W\cap Q)|]=O(d^{2.5}\log^2(m)/\sigma^2).
\end{equation}
\end{theorem}

Combining these two ingredients, we can prove Lemma \ref{lem:perimeter}.
\begin{proof}[Proof of Lemma \ref{lem:perimeter}] Fix $j\le m$ and recall that 
$rB_2^d\subset K_0^{(j)}\subset B_2^d$ by $\aR$. Conditioning on $\Ce$, we also have $\hdist(K^{(j)},K_0^{(j)})\le \alpha$. Thus we may invoke Lemma \ref{lem:quadrature} with $L_0=K_0^{(j)}$, $L=K^{(j)}$, $r_1=r$, $r_2=1$, and $\eta=\alpha=\Omega(\sigma\sqrt{d\log m})$ to obtain a probability measure $\nu$ on two dimensional planes $W\subset \R^d$ with the advertised properties; note that crucially $W$ depends only on $K_0$ and is independent of $K$. Let $I_\epsilon$ be a maximal collection of disjoint $(d-2)$-dimensional disks $S_\epsilon$ of radius $\epsilon$, with each $S_\epsilon$ contained in some $(d-2)$-face of $L$.
Notice that
\begin{align*}
&\int d\nu(W) \E\sum_{S<T\in\all}  [ \Ce\{|W\cap F_{ST}|\neq 0\} K_{ST}^{(j)}]\\
&= \E  \sum_{S<T\in\all} \int d\nu(W)  \Ce\{|W\cap F_{ST}|\neq 0\}  K_{ST}^{(j)}\\
&\ge \E  \sum_{S<T\in\all} \int d\nu(W)  \Ce\sum_{S_\epsilon\in I_\epsilon}\{|W\cap S_\epsilon|\neq 0\}  K_{ST}^{(j)}\quad\textrm{by the ``Moreover'' part of Lemma \ref{lem:quadrature}}\\
&\ge \E \sa{L_0}^{-1}\quadbound\cdot  \left[\Ce  \sum_{S<T\in\all} |F_{ST}| K_{ST}^{(j)}\right]\quad\textrm{by Lemma \ref{lem:quadrature}, choosing $\epsilon$ sufficiently small.}
\end{align*}
	
	The integrand in the first expression above above is at most $$m^2\cdot \E [\facets_0(W\cap K)\Ce]$$ since each set in $\binom{[m]}{d-2}$ appears as the intersection of at most $m^2$ adjacent pairs $S,T$. Therefore by Theorem \ref{thm:shadow} the first expression above is bounded above by $O(m^2d^{5/2}\log^2(m)/\sigma^2)$. Rearranging yields
	
	$$ \E\left[\Ce  \sum_{S<T\in\all} |F_{ST}| K_{ST}^{(j)}\right] \le O(m^2d^{2.5}\log^2(m)/\sigma^2)\cdot O(d^{3/2}\sigma\sqrt{d\log m}/r) \sa{K_0^{(j)}},$$
	implying the desired conclusion.
\end{proof}
\subsection{Proof of Lemma \ref{lem:quadrature}}\label{sec:quadrature}

\newcommand{\Kt}{\tilde{L}}
\newcommand{\Bkt}{\partial \Kt}
\newcommand{\Bx}{B_x}
\newcommand{\Bbx}{\partial\Bx}
We provide an explicit construction for the distribution of $W$.
Let $\Kt=L_0+2\eta B_2^d$ and note that its boundary $\Bkt$ is smooth; let $\psi$ be the $d-1$-dimensional surface measure on $\Bkt$. This equals both the $d-1$ dimensional Hausdorff measure and the Minkowski content of $\Bkt$.
Then let $W=V+a$ where $a$ is a point sampled according to $\psi$, and $V$ is sampled by taking the span of two Gaussian vectors (or any radially symmetric random vectors).  In order to compute $\P(V+a\cap\Se\neq\emptyset)$, it will help to first reduce it to the related probability $\P(W'\cap\Se\neq\emptyset)$ for $W=V+a'$ where $a'$ is sampled uniformly from the unit ball which shares a center with $\Se$.
In particular, let $x$ be the center of $\Se$ and denote $\Bx=B_2^d+x$. Let $\psi'$ be the $d-1$-dimensional Hausdorff measure on $\Bbx$.  Then we will reduce to the case of $\P(V+a'\cap\Se\neq\emptyset)$ for $a'$ sampled according to $\psi'$.
For any $z$, define the radial projection $\Pi_z$ by \[\Pi_z(y)=\frac{y-z}{\|y-z\|}+z.\]
Note that $\Pi_x$ is a bijection between $\Bkt$ and $\Bbx$ since every ray originating from $x$ intersects $\Bkt$ in exactly one point because $x$ is in the interior of $\Kt$, which is convex. 

\begin{figure}
\includegraphics[scale=0.42]{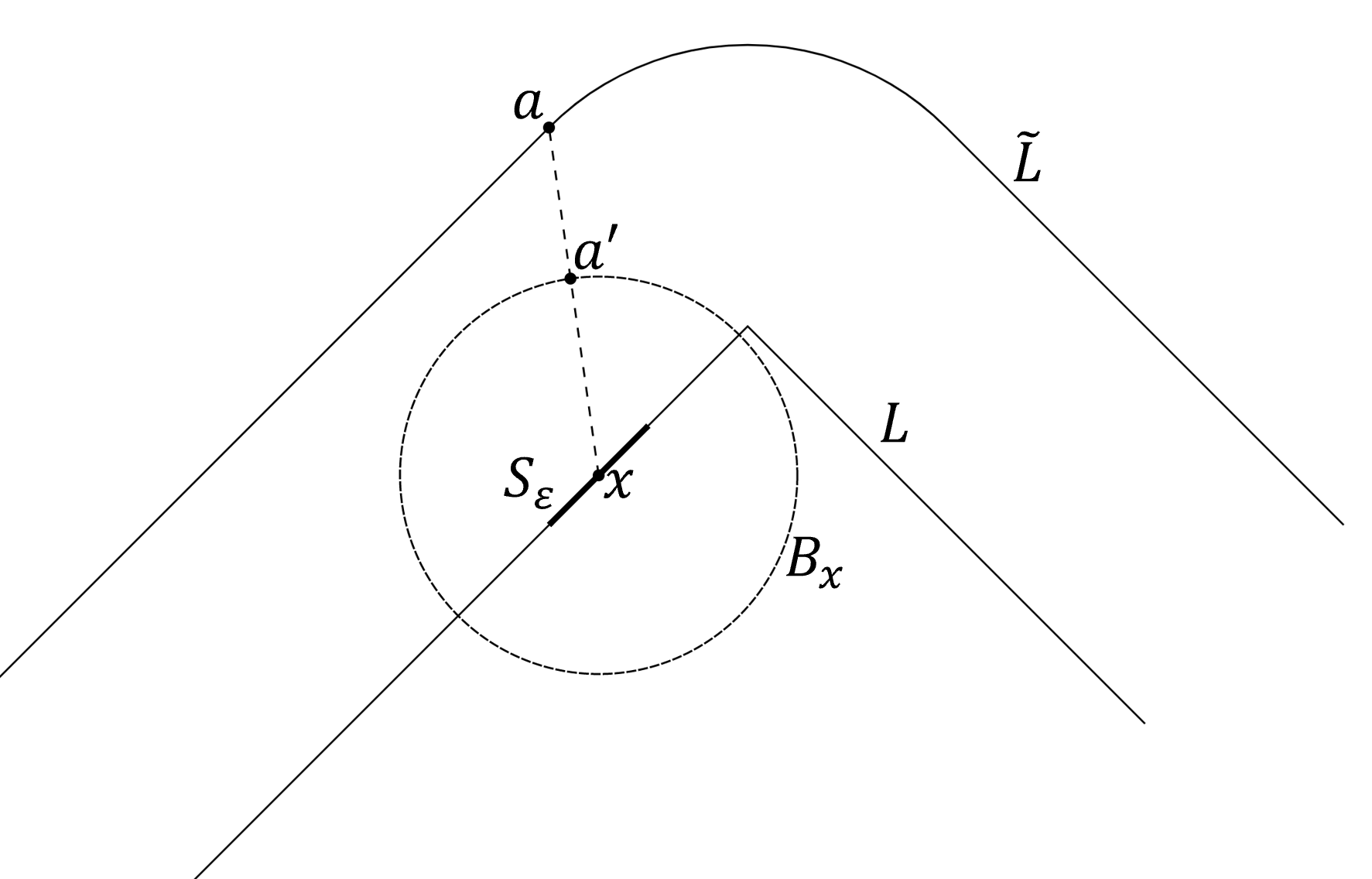}
\includegraphics[scale=0.42]{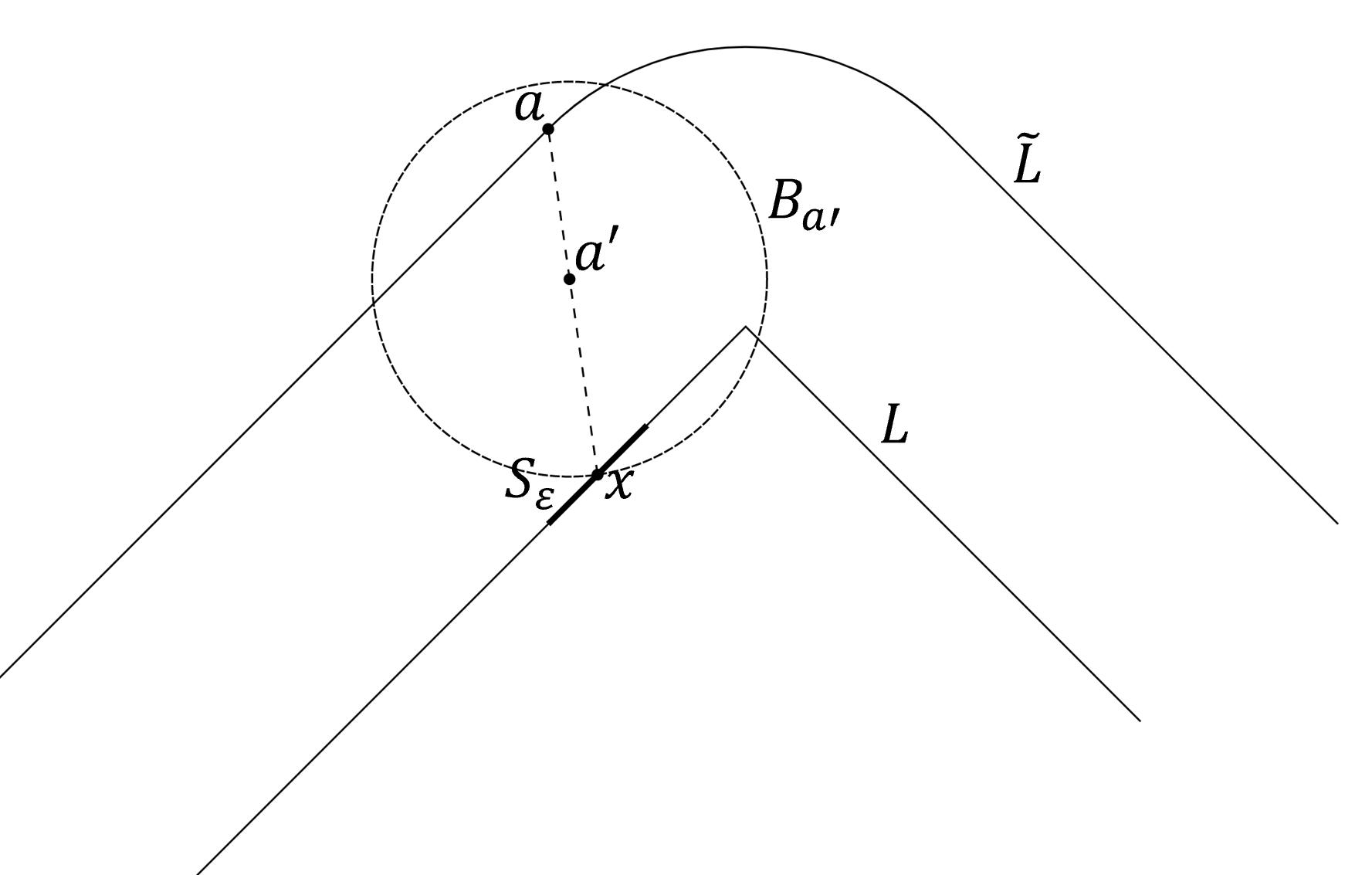}
\end{figure}

\begin{claim}
\label{claim:jac}
The push-forward of $\psi$ by $\Pi_x$ is absolutely continuous with respect to $\psi'$ with Radon-Nikodym derivative
\[
\frac{d(\psi\circ\Pi_x^{-1})(a')}{d\psi'(a')}
=
\frac{\sin \phi}{\|x-a\|^{d-1}},\qquad a'=\Pi_x(a)\in\Bbx
\]
where $\phi$ is the angle in $[0,\pi]$ between the tangent plane to $\Bkt$ at $a$ and the line segment $\overline{xa}$.
\end{claim}
\begin{proof}
An explicit Jacobian calculation given the definition of $\Pi_x$ and smoothness of $\Bkt$ gives the result.
\end{proof}
\begin{lemma}
\label{lem:circle_intersection}
Let $z\not\in\aff(\Se)$ be a point such that $\Pi_z$ is injective on $\Se$.
Let $V$ be a random two-dimensional subspace.  Then
\[
\P(V+z\cap S\neq0)=\mu(\Pi_z(S))/A_{d-2}
\]
where $\mu$ is the Hausdorff measure of $\Pi_z(\aff(S))$ and $A_{d-2}=\mu(\Pi_z(\aff(S)))$ (half the surface area of $S^{d-2}$).
\end{lemma}
\begin{proof}
Since $\aff(\Se)$ misses $z$, we have that
\(
\aff(\{z\}\cup \Se)
\) is $d-1$ dimensional.  On the other hand, $\Pi_z$ is smooth and injective on $\aff(\Se)$ so $\Pi_z(\Se)$ itself is $d-2$ dimensional.
Condition on
$(V+z)\not\subset\aff(\{z\}\cup S)
$
, which occurs with probability $1$.
Then
$(V+z)\cap
\aff(\{z\}\cup S)
$ is a line through $z$.  By symmetry, the intersection of that line with $B_2^d+z$ will be a uniformly random antipodal pair.  Exactly one point from each pair will fall in $\Pi_z(\aff(S))$.  Thus, the event we care about is the event that $y\in\Pi_z(S)$ where $y$ is sampled uniformly from $\mu$.
\end{proof}
The following Lemma takes $a$ and $a'$ to be fixed, and depends only on the randomness of $V$.
\begin{lemma}\label{claim:aax}
Let $a$ be a point not in $\aff(\Se)$ and $a'=\Pi_x(a)$.
Let $\theta$ be the angle between $\Se$ and the ray emanating from $a$ through $x$.
Then, for a uniformly random $2-$plane $V$,
\[
{\P(V+a\cap\Se\neq\emptyset)}
=
\frac{\vol_{d-2}(S_\epsilon)}{A_{d-2}}\frac{\cos\theta}{\|x-a\|^{d-2}}(1+O(d\epsilon))
\]
and
\[
{\P(V+a'\cap\Se\neq\emptyset)}
=
\frac{\vol_{d-2}(S_\epsilon)}{A_{d-2}}(\cos\theta)(1+O(d\epsilon)).
\]
where the convergence is uniform in $a,a'$.  In particular, the ratio of the above two quantities is $\|x-a\|^{d-2}$.
\end{lemma}
\begin{proof}
We apply Lemma \ref{lem:circle_intersection} twice, both times with $\Se$ playing the role of $S$.  The first time we take $a$ to play the role of $z$, and the second time $a'$.  This gives
\[
{\P(V+a\cap \Se\neq\emptyset)}
=
{\mu_a(\Pi_a(\Se))}/A_{d-2}
\quad\text{and}\quad
{\P(V+a'\cap \Se\neq\emptyset)}
=
{\mu_{a'}(\Pi_{a'}(\Se))}/A_{d-2}
\]
where $\mu_a,\mu_{a'}$ are the Hausdorff measures on $\Pi_a(\aff(\Se)),\Pi_{a'}(\aff(\Se))$ respectively.
Let $\mu'$ be the surface measure on $\aff(\Se)$.  Then the Radon-Nikodym derivatives of $\mu'$ and the pull-backs of $\mu_a$ and $\mu_{a'}$ are
\[
\frac{d(\mu_a\circ\Pi_a)(y)}{d\mu'(y)}=\frac{\cos\theta^a_y}{\|y-a\|^{d-2}}
\quad\text{and}\quad
\frac{d(\mu_{a'}\circ\Pi_{a'})(y)}{d\mu'(y)}=
\frac{\cos\theta^{a'}_y}{\|y-a'\|^{d-2}}
\]
where $\theta^a_y,\theta^{a'}_y$ are the angles between $\Se$ and the rays from $a,a'$ to $y$ respectively.  This allows us to compute
\[
\mu_a(\Pi_a(\Se))=(\mu_a\circ\Pi_a)(\Se)
=\int_{\Se}\frac{\cos\theta^a_y}{\|y-a\|^{d-2}}d\mu(y)
=\vol_{d-2}(S_\epsilon)\frac{\cos\theta^a_x}{\|x-a\|^{d-2}}(1+O(d\epsilon)).
\]
The same is true for $a'$ in place of $a$. Note that $\theta^{a'}_x=\theta^a_x=\theta$, and that $\|x-a'\|=1$.  That gives the desired result.
\end{proof}

\begin{lemma}[Reduction to $\Bbx$]\label{lem:reduction} Let $W$ be as above and let $W'$ be a uniformly random two dimensional plane through a uniformly random point $a'$ chosen from $\Bbx$. Then for sufficiently small $\epsilon>0$ (depending only on $L_0$):
$$ \Vol_{d-1}(\Bkt)\P [W\cap \Se\neq\emptyset]\ge \frac{r_1}{8r_2\eta}\Vol_{d-1}(\Bbx)\P[W'\cap \Se \neq\emptyset]
$$
\end{lemma}
\begin{proof}
Note that $a,a'$ miss $\aff(\Se)$ with probability $1$, so we implicitly condition on that event in the following.
\begin{align*}
    & \Vol_{d-1}(\Bkt) \P [W\cap \Se\neq\emptyset] \\
    &=\int \P [W\cap\Se\neq\emptyset \big| a]d\psi(a)
     \\&= \int \P [W\cap\Se\neq\emptyset \big| a'=T(a)]d(\psi\circ\Pi_x^{-1})(a')\qquad\textrm{by invertibility of $\Pi_x$ (\ref{claim:jac})}\\
     &=\int \P [W\cap\Se\neq\emptyset \big| a']\frac{\sin\phi}{\|x-a\|^{d-1}}d\psi'(a')\quad\textrm{by Claim \ref{claim:jac}}\\
     &\ge\int \left(\P [W'\cap\Se\neq\emptyset \big| a'](1/2)\right)\frac{\sin\phi}{\|x-a\|}d\psi'(a')\quad\textrm{by Claim \ref{claim:aax}, for sufficiently small $\epsilon$}\\
     &\ge\frac{r_1}{8r_2\eta} \int \P [W'\cap\Se\neq\emptyset \big| a']d\psi'(a') = \frac{r_1}{8r_2\eta}\Vol_{d-1}(\Bbx)\P[W'\cap \Se \neq\emptyset],
\end{align*}
where in the final inequality we have used $\|x-a\|\ge 2\eta-\eta=\eta$  and $\sin\phi\ge\frac{r_1}{4r_2}$ because $\Kt\supset L\supset (r_1-\eta)B_2^d\supset (r_1/2)B_2^d$ and $\Kt\subset (r_2+\eta)B_2^d\subset 2r_2 B_2^d$.
\end{proof}

\begin{lemma}[Intersection Probability for $\Bbx$]\label{lem:intersection}
\[
\P\left(W'\cap S_\epsilon\neq\emptyset\right)
=
\frac{\vol_{d-2}(S_\epsilon)}{A_{d-2}}\frac{C_d}{\sqrt d}(1+O(d\epsilon))
\]
for some constant $C_d=\Theta(1)$ depending on $d$.
\end{lemma}
\begin{proof}
Using iterated expectation, we can write
\[
\P\left(W'\cap S_\epsilon\neq\emptyset\right)
=
\E\left(\P\left(W'\cap S_\epsilon\neq\emptyset : a'\right)\right)
\]
where the outer expectation is over the randomness of $a'$ and inner probability over $V$.  The inner probability is given by \ref{claim:aax}
as
\[
\frac{\vol_{d-2}(S_\epsilon)}{A_{d-2}}\cos(\theta_x^{a'})(1+O(d\epsilon)).
\]
The only dependence on $a'$ is in $\cos(\theta_x^{a'})$.  However, by symmetry of the distribution of $a'$, $\theta_x^{a'}$ might as well measure the angle between a uniform random vector selected from $\Bbx$ and any fixed line.  Thus
\[\E[\cos\theta_x^{a'}]=C_d/\sqrt d,\]
for some constant $C_d=\Theta(1)$ depending on $d$.
\end{proof}

We can now complete the proof of Lemma \ref{lem:quadrature}. Combining Lemmas \ref{lem:reduction} and \ref{lem:intersection}, we have for sufficiently small $\epsilon>0$:

$$ \Vol_{d-1}(\Bkt)\P [W\cap S_\epsilon]\ge \frac{r_2}{2r_1\eta}\vol_{d-1}(\Bbx)\cdot \frac{\Omega(1)}{\sqrt{d}A_{d-2}}\vol_{d-2}(S_\epsilon)=\quadbound\vol_{d-2}(S_\epsilon)$$
since $\vol_{d-2}(\Bbx)/A_{d-2}=2\pi/d$, as desired.

\subsection{Removing Assumptions \aS,\aR}\label{sec:roundedness}
In this section we explain how any instance of the smoothed unit LP model may be reduced to one for which $\aS,\aR$ hold with parameter \eqref{eqn:rvalue}, incurring only a polynomial loss in $m$. 
\begin{proof}[Proof of Theorem \ref{thm:q2}] 

The idea is to
add the noise vector $g_j$ as the sum of two independent Gaussians $g_{j, 1} \sim N(0, \sigma_1^2)$ and $g_{j, 2} \sim N(0,\sigma_2^2)$ with $\sigma_1$ guaranteeing roundedness and $\sigma_2$ supplying the necessary anticoncentration and concentration for the main part of the proof. Given $\sigma<1/d$, set 
$$\sigma_1=m^8\sigma_2$$ and $\sigma_1^2+\sigma_2^2=\sigma^2$ and let $K_1$ be equal to $K_0$ perturbed by $g_1$ only. Applying Lemma \ref{lem:smoothrounded} to each $K_0^{(j)}$ and taking a union bound, we have
$$ K_1^{(j)}\supset rB_2^d\quad\forall j\le m,\qquad r=\Omega(\sigma m^{-5})=\Omega(\sigma_2 m^3),$$
with probability $1-O(m^{-2})$. Since $\sigma<1/d$, another union bound reveals that
$$K_1\subset 2B_2^d$$
with probability $1-O(m^{-2})$; let $K_2=K_1/2$. Now $K_2$ is an instance of the smoothed unit LP model, $(K_2,\sigma_2)$ satisfy $\aR$ with $r=\Omega(\sigma_2 m^3)=\Omega(\sigma/m^5)$, and
$$ 6\sqrt{d\log m}\sigma_2=o(r/m^2),$$
so $(K_2,\sigma_2)$ also satisfy $\aS$, establishing \eqref{eqn:rvalue} with the role of $(K_0,\sigma)$ now played by $(K_2,\sigma_2)$.

Invoking Theorem \ref{thm:5.15}, we conclude that with probability $1-1/m^2$, for every $\phi\in (0,1)$ there is a subset $G\subset\Omega$ with $\pi(G)\ge (1-\phi)\pi(\Omega)$ and facet diameter
$$ \tilde{O}(m^3d^8/(\sigma/m^8)^2(\sigma/m^5)\phi)=\poly(m,d)/\sigma^3\phi.$$
Moreover, by Lemma \ref{lem:degeneracy}(3), we have
$$\qu(G)\le 2\pi(G)\le 2\phi\cdot \pi(\Omega)\le \phi\cdot O(m^5/\sigma)\qu(\Omega),$$
so we conclude that $\qu(G)\ge (1-\psi)\qu(\Omega)$ for $\psi = O(m^5\phi/\sigma)$. Rewriting the diameter bound in terms of $\psi$ yields the desired conclusion. The probability may be upgraded to $1-1/\poly(m)$ by Remark \ref{rem:polyprob}
\end{proof}

\begin{lemma}[Roundedness of Smoothed Polytopes]\label{lem:smoothrounded} Suppose we have $m \geq d+1$ points $a_1, \dots, a_m \in \R^d$, and these are perturbed to $v_1, \dots, v_m$ by adding independent $g_j \sim N(0, \sigma_1^2 I_d)$ to each respective $a_j$.
Then, with probability at least $1 - O(m^{-3})$,  the convex hull $K$ of $v_1, \dots, v_m$
contains a ball of radius $r_{in} \geq \Omega(\sigma_1 m^{-5}).$\end{lemma}
\begin{proof}
Without loss of generality, taking the first $d+1$ points $a_i$, we may assume that $m = d+1$.
Then $K$ is the convex hull of $d + 1$ points $v_1, \dots, v_{d+1}$. The probability that the affine span of these points equals $\R^d$ is $1$. Let $r_{in}$ be the inradius of $K$; by Lemma \ref{lem:inradius}, we have
$$r_{in} \geq \frac{\min_i\dist(v_i, \aff(F_i))}{d+1}.$$
Let us now fix an $i$ and obtain and obtain a probabilistic lower bound on $\frac{\dist(v_i, \aff(F_i))}{d+1}.$
Reorder the points (if necessary) so that $i = d + 1$. It now follows that given the the affine span $A$ of the points $v_1, \dots, v_d$ and given $a_{d+1}$, the distribution of $\dist(v_{d+1}, A)$ is the same as the distribution of $|\tilde{g} + \dist(a_{d+1}, A)|$, where $\tilde{g} \sim N(0, \sigma_1^2)$ has the distribution of a one dimensional Gaussian with variance $\sigma_1^2$. However, the probability that $|\tilde{g} + \dist(a_{d+1}, A)|$ is less than $\sigma_1 m^{-4}$ is at most $O\left(m^{-4}\right)$. Therefore, by the union bound, 
$$\mathbb{P}\left[\min_i \dist(v_i, \aff(F_i)) > \sigma d^{-4}\right] > 1 - O\left(m^{-3}\right).$$
It follows that $$\mathbb{P}\left[r_{in} > \frac{\sigma m^{-4}}{d+1}\right] > 1 - O\left(m^{-3}\right),$$
as desired.
\end{proof}

\subsection*{Acknowledgments}
We thank Daniel Dadush, Bo'az Klartag, and Ramon van Handel for helpful comments and suggestions on an earlier version of this manuscript. We thank Ramon van Handel for pointing out the important reference \cite{izmestiev2010colin}. We thank the IUSSTF virtual center on ``Polynomials as an Algorithmic Paradigm'' for supporting this collaboration.

\bibliographystyle{alpha}
\newcommand{\etalchar}[1]{$^{#1}$}

\begin{appendix}
\section{Proof of Theorem \ref{lem:one_positive_eig}}\label{app:proof}

\begin{proof}
Let $\p(\slacks)=\{x\in\R^{d}: \A x\le\slacks\}$ for $\A\in\R^{N\times d}$ and $c>0$.  Let $m_i$ be the $i$th row of $\A$ and $F_i(\slacks)$ be the volume of $\p(\slacks)\cap\{x:\herm{m_i}{x}=\slacks_i\}$, the $i$th facet of $\p$.
Pick $\A$ to be \textit{minimal} with respect to $\slacks$ in the sense that for each $i$ there exists a point $x$ with $\herm{m_i}{x}=b_i$ but $\herm{m_j}{x}<b_j$ for $j\neq i$; note that this implies $F_c(c)>0$ for each $i\le N$.
Let $$R(c):=\text{diag}(\slacks_1/|F_1(c)|,\cdots,\slacks_N/|F_N(c)|)$$ and $\tilde{H}(c)=R^{1/2}(c)H(K(c))R^{1/2}(c)$.
By Sylvester's inertial law, it suffices to prove the statement for $\tilde H$ instead of $H$.

If $\p(\slacks)$ is simple, then the proof of the statement can be found in \cite[Proposition 3]{cordero2019one}.  In fact, it is shown there that  for simple polytopes the positive eigenvalue of $\tilde{H}(c)$ is equal to $d-1$.
If $\p(\slacks)$ is not simple, let $u$ be a vector with independent entries sampled uniformly from $[0,1)$.
We will show for each $\delta>0$ small enough that $\p(\slacks+\delta u)$ is simple with probability $1$.  Then we will show that the entries of $\tilde H(\slacks)$ are continuous in a neighborhood of $\slacks$ so that for each $\epsilon$, picking $\delta$ small enough guarantees 
\[|\tilde H(\slacks)-\tilde H(\slacks+\delta u)|\le\epsilon\]
entry-wise.

If we have both these claims, then we will have a sequence of matrices of the form
\[\tilde H(\slacks) + \left(\tilde H(\slacks+\delta u) - \tilde H(\slacks)\right)\]
which approaches $\tilde H(\slacks)$ with each matrix in the sequence having exactly one positive eigenvalue, which is $d-1$.  Since the spectrum is a continuous function of the matrix, we conclude $\tilde H(\slacks)$ has exactly one positive eigenvalue, which is $d-1$.

\textit{$\p(\slacks+\delta u)$ is simple}: 
For any collection $S\subset [n]$ if indices, let $\A_S$ (resp. $\slacks_S$, $u_S$) be the submatrix of $\A$ (resp. $\slacks$, $u$) with row $i$ included if and only if $i\in S$.
Fix any $S$ with $|S|>d$.  Then $\text{rk}(A_S)\le d$, so the column space of $A_S$ has measure $0$ in $\R^{|S|}$, meaning it misses $\slacks_S+u_S$ with probability $1$.

\textit{$\tilde H(\slacks)$ is continuous}:
It suffices to argue that $H(\p(c))$ and $R(c)$ are each continuous.  First we argue that $K(c)$ has $N$ facets everywhere in a neighborhood of $c$.
Since $\A$ is is minimal for $\slacks$, the facets of $\p(\slacks)$ correspond exactly to the rows of $\A$.
Perturbations of $\slacks$ do not change the rows of $\A$, so this correspondence is preserved if $\A$ is minimal for $\slacks+\delta u$.  To this end, let $x^{(i)}$ be such that $\herm{m_i}{x^{(i)}}=\slacks_i$ and 
$\herm{m_j}{x^{(i)}}<\slacks_j$ for $j\neq i$. Then simply pick
\[
\delta<\frac12
\frac{\min_i\min_{j\neq i}\left(\slacks_j-\herm{m_j}{x^{(i)}}\right)}
{\max_{i,j}|\herm{m_j}{m_i}|/\|m_i\|^2}
\]
and let \[\tilde x^{(i)} = x^{(i)} + u_i\delta\frac{m_i}{\|m_i\|^2}.\]
Then $\herm{m_i}{\tilde x^{(i)}} = \slacks_i+u_i\delta$
but
\[\herm{m_j}{\tilde x^{(i)}} = \herm{m_j}{x^{(i)}}+u_i\delta\frac{\herm{m_j}{m_i}}{\|m_i\|^2}
<\herm{m_j}{x^{(i)}} + \frac12(\slacks_j - \herm{m_j}{x^{(i)}}) < \slacks_j < \slacks_j+u_j\delta
\]
as desired.  This also rules out discontinuities in $|F_i(c)|$, so $R(c)$ is indeed continuous.

Now for $H(\p(\slacks))$.
For pairs $i\neq j$ such that $m_i$ is a multiple of $m_j$, we have $F_{ij}=0$ independently of $\slacks$.
Otherwise, note that
$\csc(\theta_{ij}),\cot(\theta_{ij})$ are finite and depend only on $\A$, so it suffices to show that $F_{ij}$ is continuous in $\slacks$.  To this end, $F_{ij}$ itself is the volume of a $d-2$ dimensional polytope which after rotation is of the form $\{x'\in\R^{d-2}\st\A'x'\le f(\slacks)\}$ where $M'\in\R^{(N-2)\times (d-2)}$ and $f:\R^N\to\R^{N-2}$ is an affine function of $\slacks$.
Specifically, take $i=1$ and $j=2$ without loss of generality and let $Q$ be the product of two Householder transformations such that
\[
\A Q=\begin{bmatrix}
\alpha\\
\beta&\gamma\\
v_1 & v_2 & \A'
\end{bmatrix}
\]
for some scalars $\alpha,\beta,\gamma$,  $v_1,v_2\in\R^{N-2}$, and $\A'\in\R^{(N-2)\times(d-2)}$.
Let $Q^Tx=(x_1', x_2', x_3')^T$ where $x_1',x_2'$ are scalars and $x_3'$ contains the remaining $d-2$ entries of $Q^Tx$.
Let $\slacks=(\slacks_1, \slacks_2,\slacks_3)^T$ where $\slacks_1,\slacks_2$ are scalars and $\slacks_3$ contains the remaining $N-2$ entries of $\slacks$.  Then
\[\p(\slacks)=\{x\in\R^d\st\A x\le\slacks\}=\{x\st\A QQ^Tx\le \slacks\}=\left\{x\st\begin{bmatrix}
\alpha\\
\beta&\gamma\\
v_1 & v_2 & \A'
\end{bmatrix}
\begin{bmatrix}x_1'\\x_2'\\x_3'\end{bmatrix}
\le
\begin{bmatrix}\slacks_1\\\slacks_2\\\slacks_3\end{bmatrix}
\right\}.\]
Restricting to the points where the first two constraints tight
means setting $x_1'=\slacks_1/\alpha$ and $x_2'=(\slacks_2-\beta x_1)/\gamma$.  We therefore have
\[F_{12}=\vol(\{x_3'\,|\,\A'x_3'\le \slacks_3-x_1'v_1-x_2'v_2\})\]
which is continuous in $\slacks$.


\end{proof}

\end{appendix}

\end{document}